%% file: Flux.tex
\documentclass[11pt, a4paper]{amsart}

\usepackage{amsmath}
\usepackage{tikz}
\usepackage{amssymb}

\usepackage{amsthm}%%% it must be load
\usepackage{tikz-cd}
\usepackage[all]{xy}
\usepackage{amsfonts}
\usepackage{mathrsfs}%%% to use \mathscr
\usepackage[scr=boondox, cal=esstix]{mathalpha}%%% for lower case. See https://tex.stackexchange.com/questions/668241/lowercase-mathscr-and-mathcal-letters.
\usepackage{hyperref}%%% for hyperlinks
\usepackage{xcolor}
\usepackage{MnSymbol}% provides the ``end of def'' symbol
\usepackage{mathtools}
\usepackage{subcaption}

\usepackage{comment}
\input{header_labeling}
\input{header_newcommand}

\title{The Euler Class and Flux Homomorphisms under Non-Orientability}
\author[K. Kim]{KyeongRo Kim}
\address{\hskip-\parindent
 School of Mathematics, Korea Institute for Advanced Study (KIAS), Seoul, 02455, Korea}
\email{kyeongrokim14@gmail.com}

\author{Shuhei Maruyama}
\address{\hskip-\parindent
School of Mathematics and Physics, College of Science and Engineering\\
Kanazawa University\\
Kakuma-machi, Kanazawa, Ishikawa, 920-1192, Japan}
\email{smaruyama@se.kanazawa-u.ac.jp}

\date{\today}

\begin{document}
\maketitle
\begin{abstract}
    For an orientable surface with an area form, there are two invariants of area-preserving dynamics, the flux homomorphism and the Calabi invariant.
    Tsuboi found a remarkable connection between the Calabi invariant on the closed disk and a topological invariant---the Euler class.
    In this paper, we investigate a relationship between the Euler class and the flux homomorphism for non-orientable compact surfaces with one boundary component.
    Furthermore, we prove the simplicity of the kernel of the flux homomorphisms in this non-orientable setting, which implies the non-existence of invariants analogous to the Calabi invariant.
\end{abstract}

\section{Introduction}

For a compact manifold $X$ with non-empty boundary $\partial X$, it is interesting to study the relationship between the group action on $X$ and the induced action on $\partial X$. 
In other words,  we may consider the extension problem of  a given group action on $\partial X$ to an action on $X$ in a certain regularity, e.g. smooth, orientation-preserving, volume-preserving and so on.
This theme has been studied by various authors (\cite{Ghys91}, \cite{Tsuboi00}, \cite{MannNariman20}, \cite{ChenMann23} and references therein).

Following them, we consider the boundary-restriction map $p : \Diff_0(X) \to \Diff_0(\partial X)$, where $\Diff_0(Z)$ denotes the identity component of the group $\Diff(Z)$ of diffeomorphisms on a manifold $Z$.
In general, $p$ does not admit a section homomorphism, e.g., when $X$ is a closed $n$-dimensional disk $D^n$ \cite{Ghys91}.
One of the interesting observations in \cite{Ghys91} is that $p$ admits a section homomrophism when $X$ is a M\"obius band. 
Hence,  any orientation-preserving group action $G \to \Diff_0(S^1)$ extends to an action $G \to \Diff_0(M)$ by composing the section homomorphism.
In this paper, motivated by this observation, we take non-orientability into account in studying the extension problem of surface group actions on the circle $S^1$. 
We also investigate how non-orientability affects the group structure of diffeomorphism groups.

In the case of smooth extensions on the disk $D^2$, Bowden \cite{bowden11} showed that any surface group action $\pi_1(\Sigma_g)\to \Diff_0(S^1)$ can be extended to an action $\pi_1(\Sigma_{g+g'})\to \Diff_0(D^2)$ after stabilization,
even though there is no section homomorphism of $p:\Diff_0(D^2)\to \Diff_0(S^1)$ \cite{Ghys91}.
In contrast, Tsuboi \cite{Tsuboi00} showed that a surface group action $\rho:\pi_1(\Sigma_g)\to \Diff_0(S^1)$ can be extended to an area-preserving action $\tilde{\rho}:\Diff_0(S^1)\to \Diff_\omega(D^2)$
only if the Euler number of $\rho$ is $0$, 
where  $\omega \in \Omega^2(D^2)$ is the standard area form on $D^2$ and $\Diff_{\omega}(D^2)$ is the group of $\omega$-preserving diffeomorphisms of $D^2$.
Indeed, in \cite{Tsuboi00}, he proved the transgression formula, which describes a relationship between the Euler class $\eu$ of $\Diff_0(S^1)$ and the Calabi invariant.
Recall that the Calabi invariant $\Cal$ is an $\RR$-valued surjective homomorphism on $\Diff_{\omega}(D^2)_{\mathrm{rel}} = \ker(\Diff_{\omega}(D^2) \to \Diff_0(S^1))$ defined as
\[
    \Cal(g) = \int_{D^2} \eta \wedge g^*\eta
\]
for $\eta \in \Omega^1(D^2)$ satisfying $d\eta = \omega$.

\begin{thm}[\cite{Tsuboi00}]
    Let $a_1, b_1, \ldots, a_g, b_g$ be the standard generator of the surface group $\pi_1(\Sigma_g)$, and $\rho \colon \pi_1(\Sigma_g) \to \Diff_0(S^1)$ a homomorphism.
    Let $\phi_i$ (resp. $\psi_i$) be an $\omega$-preserving diffeomorphism of $D^2$ satisfying $p(\phi_i) = \rho(a_i)$ (resp. $p(\psi_i) = \rho(b_i)$).
    Then, the Euler number of $\rho$ coincides with 
    \[
        -\frac{1}{\pi^2}\Cal([\phi_1, \psi_1] \cdots [\phi_g, \psi_g]).
    \]
\end{thm}

In \cite{bowden11}, Bowden generalized Tsuboi’s transgression formula in terms of the five-term exact sequence in group cohomology.

\begin{thm}[{\cite[Theorem 6.3]{bowden11}}]
    There exists a representative $\chi$ of the Euler class $\mathrm{eu} \in H^2(\Diff_0(S^1);\RR)$ and an $\RR$-valued $1$-cochain $C$ on the group $\Diff_{\omega}(D^2)$ of area-preserving diffeomorphisms of $D^2$ such that 
    \begin{align*}
        &C(g) = \Cal(g) \\
        &\delta C(\gamma_1, \gamma_2) = -\pi^2\chi(\gamma_1, \gamma_2)
    \end{align*}
    for every $g \in \Diff_{\omega}(D^2)_{\mathrm{rel}}$ and $\gamma_1, \gamma_2 \in \Diff_0(S^1)$.
    In particular, the Calabi invariant $\Cal$ transgresses to $-\pi^2 \mathrm{eu}$ with respect to the exact sequence 
    \[
        1 \to \Diff_{\omega}(D^2)_{\mathrm{rel}} \to \Diff_{\omega}(D^2) \to \Diff_0(S^1) \to 1.
    \]
\end{thm}

As mentioned above, in the case of the M\"obius band $M$, smooth extendability of a surface group action $\pi_1(\Sigma_g)\to \Diff_0(S^1)$ follows from the existence of a section homomorphism of the restriction map.
Therefore, the extension problem in an ``area"-preserving way naturally arises.

Although a non-orientable surface does not admit any area forms in the usual sense, there is a natural analogous notion, called an \emph{area density}.
Indeed, on a non-orientable manifold, we make use of  \emph{``twisted" differential forms}, which provide the right framework for extending most theorems of differential topology to the non-orientable setting. 
An \emph{area density} $\omega$ on a compact non-orientable surface $N$ is an everywhere positive twisted $2$-form.
See \refsec{preliminaries} for the precise definitions of density forms and twisted differential forms.

One might expect a transgression formula for the Euler class and the Calabi-type invariant for non-orientable surfaces.
However, we show that there is no homomorphism analogous to the Calabi invariant under non-orientability, by proving the simplicity of the kernel of some surjective homomorphism $\Fluxn: \Diff_\omega(N, \near \partial N)_0\to H^1(N,\partial N;L)$.
Here,  $H^*(N,\partial N;L)$ is the cohomology group of the twisted differential forms, that vanish on $\partial N$ (see \refsec{preliminaries}) and  $\Diff_\omega(F, \near \partial F)_0$ is the identity component (with respect to the $C^{\infty}$-topology) of the group of $\omega$-preserving diffeomorphisms of a compact surface $F$ with boundary, that fix some neighborhood of $\partial F$ pointwise.
\begin{thm}[Simplicity of $\ker(\Fluxn)$]\label{Thm:simplicity}
Let $N$ be a non-orientable, connected, compact surface with non-empty boundary, equipped with an area density form $\omega$. Then, $\ker(\Fluxn)$ is simple.
\end{thm}
More precisely, for a compact orientable surface $F$ with non-empty boundary,
the flux homomorphism $\Flux :\Diff_{\omega}(F)_{\mathrm{rel}} \to H^1(F,\partial F;L)$ is defined by 
\[
    \Flux(g) = [\eta - g^*\eta],
\]
where $\eta$ is a twisted $1$-form with $d\eta = \omega$ (see \refsec{Flux}).
When $F$ is orientable, the Calabi invariant is well-defined on $\ker(\Flux)$. 
In particular, when $F=D^2$, $\Flux$ is a trivial homomorphism and $\Cal$ is well-defined on the whole group $\Diff_\omega(D^2)_{\mathrm{rel}}$.  
Moreover, the restriction  $\Fluxn$ to $\Diff_{\omega}(F, \near \partial F)_0$ is surjective, and
the restriction of the Calabi invariant $\Cal : \ker(\Flux) \to \RR$ to $\ker(\Fluxn)$ remains surjective. 
Hence, $\ker(\Fluxn)$ is not simple, but it is well-known that $\ker(\Cal)\cap \ker(\Fluxn)$ is simple \cite{Banyaga78}.
Therefore, \refthm{simplicity} implies the non-existence of homomorphisms analogous to the Calabi invariant.

The key idea of the proof of \refthm{simplicity} is the \emph{Cell Division Trick} (\reflem{cellDivision}),  which is a generic phenomenon on non-orientable surfaces $N$. 
Roughly speaking, given a finite collection of properly embedded arcs $\{\gamma_i\}_i$ on $N$ (see \reffig{surface}), the complement $\fD$ of which is contractible, by the Poincar\'e duality (\refprop{PD}), there is a collection of closed $1$-forms $\{\lambda_i\}_i$ on $N$. 
For a closed $1$-form $\lambda$ on $N$, 
the \emph{$\lambda$-flux homomorhpism} $\Flux_{\lambda} : \Diff_{\omega}(N)_{\mathrm{rel}}  \to \RR$ can be defined as
\[
    \Flux_{\lambda}(g) = \int_{N} (\eta-g^*\eta) \wedge \lambda
\]
where $d\eta=\omega$.
We can observe that the $\lambda_i$-flux homomorphism measure the amount of signed area passing through $\gamma_i$ along an isotopy from the identity to $g$   (see \refsec{swept_area}).
Hence, the movement of an element in $\ker(\Fluxn)$ is trapped in the ``complement" of  $\{\gamma_i\}_i$.
Therefore, the only possible movement of such an element is twisting a contractible region in $\fD$.
The classical Calabi invariant, defined on an orientable surface, measures this twisting effect.
However, the situation changes dramatically on non-orientable surfaces. 
Non-orientability cancels out the local twisting effect. 
We formulate this phenomenon in the form of the Cell Division Trick. 

Even though there is no Calabi-type homomorphism, we can make use of flux homomorphisms to establish the transgression formula for non-orientable surface $N$ with one boundary component.

\begin{thm}\label{Thm:transgression_flux}
 Let $N$ be a non-orientable surface with one boundary component $\partial N=S^1$, equipped with an area density form $\omega$.
 Then, for each closed $1$-form $\lambda$ on $N$, there exists a representative $\chi$ of (non-zero constant multiple of) the Euler class $\eu \in H^2(\Diff_0(S^1);\RR)$ and an $\RR$-valued $1$-cochain $F_{\lambda}$ on $\Diff_{\omega}(N)$ such that
    \begin{align*}
        &\Flux_{\lambda}(g) = F_{\lambda}(g) \\
        &\delta F_{\lambda} (h_1, h_2) = \chi(p(h_1), p(h_2))
    \end{align*}
    for $g \in \Diff_{\omega}(N)_{\mathrm{rel}} $ and $h_1, h_2 \in \Diff_{\omega}(N)$.
    In particular, the $\lambda$-flux homomorphism transgresses to the Euler class $\mathrm{eu}$ up to a non-zero constant multiple, with respect to the exact sequence 
\begin{align}\label{eq:bdry_res_ex}
        1 \to \Diff_{\omega}(N)_{\mathrm{rel}} \xrightarrow{i} \Diff_{\omega}(N) \xrightarrow{p} \Diff_0(S^1) \to 1.
    \end{align}
\end{thm}
\begin{rmk}
    The exactness of \eqref{eq:bdry_res_ex} is shown in \refprop{boundarySurj}.
\end{rmk}
In particular, this implies the following, which is analogous to the theorem of Tsuboi.
\begin{thm}
    Let $a_1, b_1, \ldots, a_g, b_g$ be the standard generator of the surface group $\pi_1(\Sigma_g)$, and $\rho \colon \pi_1(\Sigma_g) \to \Diff_0(S^1)$ a homomorphism.
    Let $\phi_i$ (resp. $\psi_i$) be an $\omega$-preserving diffeomorphism of $N$ satisfying $p(\phi_i) = \rho(a_i)$ (resp. $p(\psi_i) = \rho(b_i)$).
    Then, the Euler number of $\rho$ coincides with 
    \[
        \Flux_{\lambda}([\phi_1, \psi_1] \cdots [\phi_g, \psi_g]).
    \]
    up to a non-zero constant multiple.
\end{thm}

\subsection*{Origanization}
\refsec{preliminaries} is for preliminaries, including differential and algebraic topology for twisted forms.
Also, we remark some facts in Appendix~\ref{Appendix}.
In \refsec{transgression}, we prove the transgression formula for the flux homomorphism and the Euler class (\refthm{transgression_flux}).
In \refsec{swept_area}, we explain a geometric meaning of the flux homomorphism in terms of swept-area.
In \refsec{excision}, we show a fragmentation lemma (\reflem{frag}) for $\ker(\Fluxn)$. Then, by using  the Cell division trick (\reflem{cellDivision}), we promote it to a strong version of fragmentation lemma (\reflem{strongFrag}).
Based on \reflem{strongFrag}, we prove the simplicity theorem (\refthm{simplicity}).

\section{Preliminaries}\label{Sec:preliminaries}

\subsection{Twisted differential forms}
We recall some basic notions related to twisted differential forms. For details, e.g. Stokes' Theorem, elementary exterior algebra for twisted differential forms and so on, see \cite{BottTu82} and \cite{deRham84}, or Appendix~\ref{Appendix}.

Let $X$ be a compact manifold with or without boundary and $\{(U_\alpha,\varphi_\alpha)\}$ its atlas.
We denote by $s_{\alpha \beta}$ the sign of the Jacobian determinant of $\varphi_\alpha\circ \varphi_\beta^{-1}$ (if it is defined).
The \emph{orientation bundle} $L_X$ of $X$ is a line bundle over $X$, whose atlas $\{(V_\alpha,\psi_\alpha)\}_{\alpha\in \Gamma}$ is defined as follows:
\begin{itemize}
    \item $V_\alpha=L_X|_{U_\alpha}$ and  $\psi_\alpha:V_\alpha \to U_\alpha \times \RR$ is a local trivialization of $L_X$;
    \item if $V_\alpha \cap V_\beta\neq \varnothing$, then  $\psi_\alpha\circ \psi_\beta^{-1}:\psi_\beta(V_\alpha \cap V_\beta)\to \psi_\alpha(V_\alpha \cap V_\beta)$ is given as 
    \[
    \psi_\alpha\circ \psi_\beta^{-1}(x,v)=(\varphi_\alpha\circ \varphi_\beta^{-1}(x),s_{\alpha \beta}(x)v )
    \]
    for $(x,v)\in \psi_\beta(V_\alpha \cap V_\beta)\subset U_\beta \times \RR$.
\end{itemize}
Namely, the transition functions of $L_X$ are given by $s_{\alpha\beta}$.
If there is no confusion, we just write $L$ for $L_X$.

A \emph{twisted differential $p$-form} (or simply, a \emph{twisted $p$-form}) on $X$ is a global section of the vector bundle $(\bigwedge^pT_X^*)\bigotimes L$.
We denote by $\Omega^p(X;L)$ the space of twisted $p$-forms on $X$.
The \emph{exterior derivative } $d:\Omega^p(X;L)\to \Omega^{p+1}(X;L)$ is defined as follows:
For each $\alpha\in \Gamma$, a local section $e_\alpha$ of $L$ over $U_\alpha$ is given as  $e_\alpha(u)=(u,1)$ for all $u\in U_\alpha$. We call such a section a \emph{standard locally constant section}.
In $(U_\alpha,\varphi_\alpha)$, a twisted $p$-form $\mu$ can be written as
$\mu=\nu \otimes e_\alpha$ for some differential $p$-form $\nu$ over $U_\alpha$.
Then, we set $d\mu=(d\nu)\otimes e_\alpha$ and assume that $d$ satisfies the linearity and the Leibnitz rule.

Since $d^2=0$, we have a well-defined cochain complex
$(\Omega^\bullet(X;L), d)$, called the \emph{twisted de Rham complex}, and its homology $H^*(X;L)$.
Moreover, we denote by $\Omega^\bullet_c(X;L)$ the space of compactly supported twisted forms and define the homology $H_c^*(X;L)$ for the cochain complex
$(\Omega^\bullet_c(X;L), d)$.

\begin{const}[Associated forms]\label{Const:associatedForm}
    If a $p$-form $\nu$ on $X$ is supported on an orientable submanifold $Z$, then we can associate $\nu$ with a twisted $p$-form $\mu$ on $X$ as follows.
    For each component $Z_j$ of $Z$, there is a differential $p$-form $\nu_j$ on $X$ such that $\nu_j = \mu$ on $Z_j$ and $\nu_j$ is supported on $Z_j$.
    Take a constant section $e_j$ of $L_X|_{Z_j}$ such that either $\psi_\alpha\circ e_j=1$ or $\psi_\alpha\circ e_j=-1$ in any local trivialization $(\psi_\alpha, V_\alpha)$.
    Then, $\nu_j \otimes e_j$ is a twisted $p$-form supported on $Z_j$, which can be extended to $X$ by defining it to zero outside $Z_j$. Therefore, $\mu=\sum_j \nu_j \otimes e_j$ is a well-defined twisted form on $X$, supported on $Z$.
    Conversely, given a twisted form supported on an orientable submanifold, we can construct an ordinary differential form by undoing the tensoring.
\end{const}

Since there is no canonical choice of $e_j$, there is no canonical way of association.
Nonetheless, when $X$ is orientable, $L$ is the trivial line bundle, we can take a global constant section of $L$ to associate each form to a twisted form as above. This implies that for an orientable manifold, the twisted cohomology is the same as the ordinary de Rham cohomology.

In fact, twisted forms are nothing but differential forms with the following coordinate change rule: after changing the coordinate in the usual sense, we multiply by the sign of the Jacobian determinant of the transition map. 
This is how de Rham introduced \emph{differential forms of odd type} in \cite{deRham84}.
According to his terminology, the usual differential forms are \emph{differential forms of even type}.
Note that the wedge product of two forms of the same type is of even type, and the wedge product of two forms of different types is of odd type.

\subsection{Pullback of twisted forms}\label{Subsec:pullback}
To define pullback of twisted forms, we first introduce an orientation of a map.
Let $X$ and $Y$ be connected smooth manifolds with or without boundary and $h:X\to Y$ a smooth map.
Say that $\pi_X:L_X\to X$ and $\pi_Y:L_Y\to Y$ are the projection maps. 
An \emph{orientation} of $h$ is a bundle morphism $h^\flat:L_X\to L_Y$ such that $h\circ \pi_X= \pi_Y\circ h^\flat$ and for every pair of trivializations $(U,\varphi)$ and $(V,\psi)$ of $L_X$ and $L_Y$, respectively, with $h^\flat(U)\subset V$, if $e_V$ is the standard locally constant section of $L_Y$ over $\pi_Y(V)$, then the local section $e$ of $L_X$ over $\pi_X(U)$, given as 
\[
h^\flat(e(x))=e_V(h(x))
\]
is either the standard locally conatnst section $e_U$ of $L_X$ over $\pi_X(U)$ or $-e_U$.
Compare it with the definition of an orientation of a map in \cite[Chapter~II, $\mathsection$ 5]{deRham84}.

When $\dim(X)=\dim(Y)$ and $h$ has no critical point, we can assign the \emph{canonical orientation} $h^\flat$, which is introduced in \cite[page~21]{deRham84}, as follows:
for any local coordinates $(U_\alpha,\varphi_\alpha)$ and $(V_\beta,\psi_\beta)$ of $X$ and $Y$, respectively, such that $h(U_\alpha)\subset V_\beta$ and the Jacobian determinant of $ \psi_\beta \circ h \circ \varphi_\alpha^{-1}$ is positive on $\varphi_\alpha(U)$,
we define $h^\flat$ as
\[
h^\flat(e_\alpha(x))=e_\beta(h(x)),
\]
where $e_\alpha$ and $e_\beta$ are the standard locally constant sections over $U_\alpha$ and $V_\beta$, respectively.
Also, whenever $X$ is a submanifold of $Y$, that is, there is an inclusion map $i:X\hookrightarrow Y$, we can also have the \emph{canonical orientation} of $i$ by taking the restriction of the identity map of $L_Y$ on $i(X)$, namely, $\id_{L_Y}{\restriction_{i(X)}}$.
From now on, we will use the canonical orientations without further mention, unless confusion might arise.

Let $h:X\to Y$ be a smooth map oriented by $h^\flat$ and $\mu$ a twisted form on $Y$.
The \emph{pullback} $h^*\mu$ of $\mu$ by $h$ with respect to $h^\flat$ is defined as 
\[
(h^*\mu)_x= h^*\nu \otimes (h^\flat)^{-1}(e)
\]
for $\nu\in (\bigwedge^p T^*Y)_{h(x)}$ and $e\in L_{h(x)}$ with $\mu_{h(x)}=\nu\otimes e$. 

Now, we can introduce the relative version of a twisted de Rham complex.
Let $X$ be a compact manifold with boundary $\partial X$ and $i:\partial X \hookrightarrow X$ the inclusion map. 
We denote by $\Omega^p(X,\partial X;L)$ the space of twisted $p$-forms $\mu$ such that $i^*\mu=0$.
Then, we have a well-defined cochain complex $(\Omega^\bullet(X,\partial X;L), d)$ and its homology $H^*(X,\partial X;L)$.

\subsection{Volume-preserving diffeomorphisms}

On a non-orientable manifold, a volume form is not well-defined. 
Nonetheless, for any manifold $X$, an everywhere positive twisted 
$\dim(X)$-form is well-defined, and we call such a twisted form a \emph{volume density form}.
In particular, when $\dim(X)=2$, we call it an \emph{area density form}.
We can think of a volume density form as a natural generalization of a volume form since when $M$ is orientable, a volume density form is nothing but a volume form.

Let $\omega$ be a volume density form on a compact, connected manifold $X$ with or without boundary. 
We write $\Diff_\omega(X)$ for the group of smooth diffeomorphisms $g$, that preserve $\omega$, namely, $g^*\omega=\omega$. 
We call each element of $\Diff_\omega(X)$ a \emph{volume-preserving diffeomorphism} (or area-preserving diffeomorphism, if $\dim(X)=2$) or \emph{$\omega$-preserving diffeomorphism} on $X$.
Also, we denote by $\Diff_\omega(X,\partial X)$ the group of elements of $\Diff_\omega(X)$, fixing the boundary pointwise and by $\Diff_\omega(X,\near \partial X)$ the group of elements of $\Diff_\omega(X)$, fixing some neighborhood of $\partial X$ pointwise. 
The identity components of $\Diff_\omega(X)$,  $\Diff_\omega(X,\partial X)$, and $\Diff_\omega(X,\near \partial X)$, equipped with the $C^\infty$-topology, are denoted by $\G(X)$, $\Gr(X)$, and $\Gn(X)$, respectively. 
Indeed, the specific choice of volume density form is not important, since Moser's theorem for volume density forms can be established (see \refthm{Moser}).

\subsection{Flux homomorphisms}\label{Sec:Flux}
Let $X$ be a compact, connected $n$-manifold with non-empty boundary, possibly non-orientable, equipped with a volume density form $\omega$.
In \cite{KimMaruyama25}, the authors observed the exactness of $\omega$, namely, there is a twisted $(n-1)$-form $\eta$ such that $d\eta=\omega$. 

From now on, we assume that $n=2$.
For $\eta \in \Omega^1(X;L)$ satisfying $d\eta = \omega$, we define a map $\Flux: \Gr(X) \to H^1(X,\partial X;L)$ by
\[
    \Flux(g) = [\eta - g^*\eta].
\]
This map $\Flux$ is called the \emph{flux homomorphism} on $\Gr(X)$.
In particular, we denote by $\Fluxn$ the restriction of $\Flux$ into $\Gn(X)$. 

\begin{lem}
    The flux homomorphism is independent of the choice of $\eta$, and is a homomorphism.
\end{lem}   

The proof is the same as the case of orientable symplectic manifolds, due to \refprop{homotopyInvariance}.

\subsection{$\lambda$-flux homomorphisms}\label{Subsec:lambda-flux}
Let $F$ be a compact surface with boundary, equipped with an area density  form $\omega$ such that $\omega=d\eta$.
Since $H^1(F,\partial F;L_F)\cong H_c^1(\mathring{F};L_{\mathring{F}})$ (\refprop{relToCsupp})  and $H^1(F)\cong H^1(\mathring{F})$
, it follows from the Poincar\'{e} duality of $\mathring{F}$ (see \cite[Theorem 7.8]{BottTu82}) that 
the map $H^1(F,\partial F;L) \otimes H^1(F) \to \RR$ induced by 
\[
    \Omega^1(F,\partial F;L) \times \Omega^1(F) \ni (\alpha, \beta) \mapsto \int_F \alpha \wedge \beta \in \RR
\]
is non-degenerate.
Based on the duality, we define a \emph{$\lambda$-flux homomorphism} $\Flux_{\lambda} \colon \Gr(F) \to \RR$ by
\[
    \Flux_{\lambda} (g) = \int_F (\eta - g^*\eta) \wedge \lambda
\]
for $g \in \Gr(F)$, where $\lambda \in \Omega^1(F)$ is a closed form.
Likewise, we define $\Fluxnl:\Gn(F)\to \RR$.

\subsection{Local Calabi invariant}

Let $F$ be a compact, connected surface with boundary and $U$ a contractible open subset of $F$. 
We denote by $\Gn(F)_U$ the set of elements in $\Gn(F)$, the supports of which are contained in $U$.
Since $U$ is contractible, we can choose a constant section $e$ of $L_F|_{U}$ such that  $\psi_\alpha\circ e=1$ or $\psi_\alpha\circ e=-1$ in any local trivialization $(\psi_\alpha, V_\alpha)$.

By the Poincar\'e lemma, for each $g\in \Gn(F)_U$, there is a twisted $0$-form $f_g$, compactly supported on $U$, such that $\eta-g^*\eta=df_g$.
Then, the \emph{local Calabi invariant} on $U$ with respect to $e$, $\Cal_U: \Gn(F)_U \to \RR$, is defined as 
\[
\Cal_U(g)=\int_{U} \bar{f}_g \omega
\]
where $\bar{f}_g$ is the ordinary $0$-form associated with $f_g$ with respect to $e$, given by \refconst{associatedForm}, that is, $f=\bar{f}_g\otimes e$.
The local Calabi invariant is a well-defined homomorphism, which is surjective.
Note that if $e$ is replaced with $-e$, then the sign of $\Cal_U(g)$ is changed.

\subsection{Group cohomology and Euler class}\label{Sec:pre_eulerclass}
In this subsection, we briefly recall the notion of group cohomology.

Let $G$ be a group and $A$ an abelian group.
For $n \in \ZZ_{\geq 0}$, let $C^n(G;A)$ be the set of maps from $G^n$ to $A$ and define the \emph{coboundary map} $\delta \colon C^n(G;A) \to C^{n+1}(G;A)$ by
\[
    \delta c(g_1, \cdots, g_{n+1}) = c(g_2, \cdots, g_{n+1}) + \sum_{i = 1}^{n} (-1)^i c(g_1, \cdots, g_ig_{i+1}, \cdots, g_{n+1}) + (-1)^{n+1} c(g_1, \cdots, g_n).
\]
Here, we regard $C^0(G;A) = A$ and $\delta = 0 \colon C^0(G;A) \to C^1(G;A)$.
Then, the homology $H^*(G;A)$ of the cochain complex $(C^{\bullet}(G;A), \delta)$ is called the \emph{cohomology of the group $G$ with coefficients in $A$}.
It is easily verified that the first cohomology $H^1(G;A)$ is isomorphic to the $A$-module of all homomorphisms from $G$ to $A$.

Given an exact sequence $1 \to K \to E \to G \to 1$ of groups, we have the following \emph{five-term exaxt sequence}:
\begin{align}\label{eq:five-term}
    0 \to H^1(G;A) \to H^1(E;A) \to H^1(K;A)^{E} \xrightarrow{\tau} H^2(G;A) \to H^2(E;A).
\end{align}
Here $H^1(K;A)^{E}$ is the $A$-module of $E$-conjugation invariant homomorphisms from $K$ to $A$.
The \emph{transgression map} $\tau \colon H^1(K;A)^{E} \to H^2(G;A)$ in \eqref{eq:five-term} is given as follows.
\begin{prop}[{see \cite[(1.6.6)Proposition]{MR2392026}}]\label{Prop:transgression_general}
    Let $1 \to K \to E \xrightarrow{p} G \to 1$ be an exact sequence of groups.
    For an $E$-conjugation invariant homomorphism $x \colon K \to A$, the cohomology class $\tau(x)$ is given as follows:
    There exist a map $y \colon E \to A$ and a (uniequly determined) $2$-cocycle $c_y \in C^2(G;A)$ satisfying 
    \begin{align}\label{eq:transgression_formula_general}
        & y(k) = x(k) \nonumber \\
        & \delta y(e_1, e_2) = c_y(p(e_1), p(e_2))
    \end{align}
    for every $k \in K$ and $e_1, e_2 \in E$.
    For such $y$ and $c_y$, the equality $\tau(x) = [c_y]$ holds.
\end{prop}

\begin{rmk/}
    If a map $y \colon E \to A$ satisfies $y(ek) = y(e) + y(k) = y(ke)$ for every $k \in K$ and $e \in E$, then the coboundary $\delta y$ descends to a cocycle $c_y$ on $G$.
    In particular, this $y$ and the induced $c_y$ satisfy \eqref{eq:transgression_formula_general}.
\end{rmk/}

If $\tau(x) = z$ holds, then the class $x \in H^1(K;A)^E$ is said to \emph{transgress to $z \in H^2(G;A)$ with respect to an exact sequence $1 \to K \to E \to G \to 1$}.

\begin{ex}\label{Exa:matsumoto_cocycle}
    Let $T \colon \RR \to \RR$ be the translation by one and $\tDiff_0(S^1)$ the group of diffeomorphisms of $\RR$, that commute with $T$.
    Then, $\tDiff_0(S^1)$ gives rise to an exact sequence $0 \to \ZZ \to \tDiff_0(S^1) \to \Diff_0(S^1) \to 1$.
    Here, we regard $S^1$ with $\RR/\ZZ$, and $\ZZ$ with the subgroup $\langle T \rangle$ of $\tDiff_0(S^1)$.
    Then the inclusion homomorphism $i \colon \ZZ \to \RR$ and the Poincar\'{e} translation number $\trot \colon \tDiff_0(S^1) \to \RR$ satisfy the assumption in Proposition \ref{Prop:transgression_general}.
    Hence we have $\tau(i) = [c_{\trot}]$.
    On the other hand, the negative of $c_{\trot}$ is just the Matsumoto cocycle (\cite{MR848896}), which represents the (real) Euler class $\mathrm{eu} \in H^2(\Diff_0(S^1);\RR)$.
    Hence, we have $-\tau(i) = \mathrm{eu}$.
\end{ex}

\section{Transgression of $\Flux$}\label{Sec:transgression}
In this section, we prove \refthm{transgression_flux}. 
We first observe that for any compact, connected surface $F$ with non-empty boundary, the boundary restriction map  $\G(F)\to \Diff_0(\partial F)$ is surjective and 
\[
1\to \Gr(F)\to G(F)\to \Diff_0(\partial F) \to 1
\]
is exact.
Then, a five-term exact sequence follows from (\ref{eq:five-term}) with $A=\RR$.
To compute the formula \eqref{eq:transgression_formula_general} for $\Flux$ and $\eu$, in \refsec{euFormula}, we provide an explicit representative of the Euler class. By using this, we show \refthm{transgression_flux}.

\subsection{Surjectivity of the boundary restriction map  $\G(F)\to \Diff_0(\partial F)$}
By modifying \cite[Lemma~(2.2)]{Tsuboi00}, we can observe the following surjectivity:
\begin{prop}\label{Prop:boundarySurj}
    Let $F$ be a compact, connected surface with boundary, equipped with an area density  form $\omega$. 
    Then, the homomorphism $\G(F)\to \Diff_0(\partial F) $ defined as 
    $g\mapsto g{\restriction_{\partial F}}$ is surjective.
\end{prop}
\begin{proof}
    It is enough to show that given a vector field $\xi$ on $\partial F$, there is a divergence free vector field on $F$ tangent to $\partial F$, that is an extension of $\xi$. 
    Note that $\partial F$ is a disjoint union of copies of the circles and it is orientable. Hence, a collar neighborhood of $\partial F$ is a disjoint union of closed annulus and it is  orientable. 
    
    Fix an area density  form $\omega_\partial$ on $\partial F$. 
    There is a parameterization $i:\partial F \times [0,1]\to F$ of a collar neighborhood of $\partial F$ such that $i(F\times \{0\})=\partial F$. Indeed, by \refthm{Moser}, we may identify $i(F\times [0,1])$ with $F\times [0,1]$ and assume that the area density  form $\omega$ on  $F\times [0,1]$ takes the form  $\omega_\partial \wedge ds$ where $s$ denotes the coordinate on $[0,1]$. 

    We define $\dv(\xi):\partial F \to \RR$ as 
    $\dv(\xi)\omega_\partial=\cL_\xi \omega_\partial$.
    Now, we set a vector field $X$ on $\partial F\times [0,1]$ such that
    \[X(t,s)=\xi(t)-s\cdot\dv(\xi)(t) \frac{\partial}{\partial s}\] for $(t,s)\in \partial F\times [0,1]$.
    Note that $X(t,0)=\xi(t)$.
    Also, we have that 
    \[
    \cL_X \omega=\cL_{\xi}(\omega_\partial \wedge ds)- \cL_{s\cdot\dv(\xi)(t) \partial/\partial s}(\omega_\partial \wedge ds)= \dv(\xi) \omega_\partial\wedge ds- d(s\cdot \dv(\xi)\omega_\partial)=0
    \]
    and $X$ is a divergence free vector field on $\partial F \times [0,1]$. 

    Now, we put a twisted $0$-form $\alpha$ on $F$, supported on $\partial F \times [0,1]$, as    \[\alpha(t,s)=s\mu(t,s) \cdot (i(\xi) \omega_\partial)(t)\]
    for $(t,s)\in \partial F\times [0,1]$, where $\mu$ is a smooth function on $\partial F\times [0,1]$ that is $0$ on a neighborhood of $\partial F \times \{1\}$ and $1$ on  a neighborhood of $\partial F \times \{0\}$.
    Note that 
    \[\dv(\xi)\omega_\partial= \cL_\xi \omega_\partial= i(\xi)d\omega_\partial+d(i(\xi)\omega_\partial)=d(i(\xi)\omega_\partial)\]
    since $d\omega_\partial=0$.
    Therefore, on some neighborhood of $\partial F\times \{0\}$, we have that 
    \begin{align*}
            d\alpha&=d(s(i(\xi)\omega_\partial)(t))\\
            &=(i(\xi)\omega_\partial)ds+s\cdot d(i(\xi)\omega_\partial)\\
            &=(i(\xi)\omega_\partial)ds+s\cdot \dv(\xi) \omega_\partial\\
            &=i(X)(\omega_\partial \wedge ds)\\
            &=i(X)\omega.
    \end{align*}
    Thus, the vector field $Y$, defined by $d\alpha=i(Y)\omega$, is the desired vector field on $F$.    
\end{proof}

\subsection{Euler cocycles}\label{Sec:euFormula}

Here, we provide the explicit representative $\chi$ of the Euler class, appeared in \refthm{transgression_flux}.
Let $\phi(\theta)d\theta$ and $\psi(\theta)d\theta$ be $1$-forms on $S^1$.
For $\gamma \in \Diff_0(S^1)$, the difference $\phi(\theta)d\theta - \gamma^*(\phi(\theta)d\theta)$ %of $\phi(\theta)$ and the pullback by $\gamma$ 
is an exact form since $\Diff_0(S^1)$ is path-connected.
Let $\alpha_{\gamma}$ be a $0$-form on $S^1$ satisfying $d\alpha_{\gamma} = \phi(\theta)d\theta - \gamma^*(\phi(\theta)d\theta)$.
Define $\chi \in C^2(\Diff_0(S^1);\RR)$ by 
\begin{align}\label{eq:Euler_cocycle}
    \chi(\gamma_1, \gamma_2) = \int_{S^1} (\alpha_{\gamma_1} - \gamma_2^* \alpha_{\gamma_1})\psi(\theta)d\theta.
\end{align}
It is obvious that $\chi$ does not depend on the choice of $\alpha_{\gamma_1}$.

\begin{lem}
    The cochain $\chi$ is a cocycle and satisfies $[\chi] = AB\cdot \mathrm{eu}$, where $A = \int_{S^1} \phi(\theta)d\theta$ and $B = \int_{S^1} \psi(\theta)d\theta$.
\end{lem}
\begin{proof}
    We define $\Phi \colon \RR \to \RR$ and $F \colon \tDiff_0(S^1) \to \RR$ by 
    \[
        \Phi(\theta_0) = \int_0^{\theta_0} \phi(\theta)d\theta,
    \]
    and
    \[
        F(\tgamma) = \int_{0}^{1} \big(\Phi(\theta) - \Phi(\tgamma(\theta)) \big) \psi(\theta) d\theta.
    \]
    Say that $-AB\cdot i:\ZZ \to \RR$ is a homomorphism defined as $1\mapsto -AB$.
    Since $\Phi(\theta + n) = \Phi(\theta) + nA$ for every $n \in \ZZ$ and $\theta \in \RR$, the maps $-AB\cdot i :\ZZ \to \RR$ and $F \colon \tDiff_0(S^1) \to \RR$ satisfy the assumption in  \refprop{transgression_general}.
    Hence, the cocycle $c_F$ defined by $p^*c_{F} = \delta F$ satisfies $[c_{F}] = -AB \cdot \tau(i)$, where $p:\tDiff_0(S^1)\to \Diff_0(S^1)$ is the projection.
    Together with \refexa{matsumoto_cocycle}, we obtain $[c_{F}] = AB \cdot \mathrm{eu}$.
    
    We now prove $c_{F} = \chi$.
    %Recall that the cocycle $c_{F}$ is defined by the equality $p^*c_{F} = \delta F$.
    For $\gamma_1, \gamma_2 \in \Diff_0(S^1)$, take their lifts $\tgamma_1, \tgamma_2 \in \tDiff_0(S^1)$.
    Then we have
    \begin{align*}
        c_F(\gamma_1, \gamma_2) &= \delta F(\tgamma_1, \tgamma_2) \\
         &= \int_0^1 \left[ \big(\Phi(\theta) - \Phi\circ \tgamma_1(\theta)\big)+ \big(\Phi(\theta) - \Phi\circ \tgamma_2(\theta)\big)- \big(\Phi(\theta) - \Phi\circ(\tgamma_1\tgamma_2)(\theta) \big)\right]\psi(\theta)d\theta \\
        &= \int_0^1 \big[\Phi(\theta) - \Phi\circ \tgamma_1(\theta) - \big(\Phi\circ \tgamma_2(\theta) - \Phi\circ(\tgamma_1\tgamma_2)(\theta) \big) \big]\psi(\theta)d\theta.
    \end{align*}
    Consider the periodic function $\widetilde{\beta}_{\gamma_1} \colon \RR \to \RR$ defined by 
    \[
        \widetilde{\beta}_{\gamma_1}(\theta) = \Phi(\theta) - \Phi(\tgamma_1(\theta)) + \Phi(\tgamma_1(0)).
    \]
    This function $\widetilde{\beta}_{\gamma_1}$ does not depend on the choice of lift $\tgamma_1$, and descends to a $0$-form $\beta_{\gamma_1}$ on $S^1$.
    Moreover, $\beta_{\gamma_1}$ satisfies $d\beta_{\gamma_1} = \phi(\theta)d\theta - \gamma_1^*(\phi(\theta)d\theta)$.
    Hence, we obtain
    \begin{align*}
        \chi(\gamma_1, \gamma_2) &= \int_{S^1} (\beta_{\gamma_1} - \gamma_2^* \beta_{\gamma_1})\psi(\theta)d\theta\\ 
        &= \int_0^1 \big(\widetilde{\beta}_{\gamma_1} - \tgamma_2^* \widetilde{\beta}_{\gamma_1}\big)\psi(\theta)d\theta\\
        &=\int_0^1 \big[\Phi(\theta) - \Phi\circ \tgamma_1(\theta) - \big(\Phi\circ \tgamma_2(\theta) - \Phi\circ(\tgamma_1\tgamma_2)(\theta) \big) \big]\psi(\theta)d\theta.
    \end{align*}
    This completes the proof.
\end{proof}

\subsection{Transgression of $\Flux$}

Let $N$ be a compact, connected, non-orientable surface with one boundary component, equipped with an area density  form $\omega=d\eta$ and let $\lambda$ be a closed $1$-form in $N$. 
Let $i : S^1 = \partial N \hookrightarrow N$ be the inclusion with the canonical orientation.
Fix a global constant section $e$ of $L_{S^1}=L_N|_{S^1}$ such that either $e=1$ or $e=-1$ for any trivialization of $L_{S^1}$.
By \refconst{associatedForm}, there is a corresponding ordinary $1$-form $\mu$ in $S^1$ such that $\mu \otimes e=i^*\eta$.
For simplicity, we think of $i^*\eta$ as $\mu$.  
Set $A_{\omega} = \int_N \omega = \int_{S^1} i^*\eta$ and $B_{\lambda} = \int_{S^1} i^*\lambda$.

Let $F_{\lambda} \colon \G(N) \to \RR$ be a map, defined by  
\[
    F_{\lambda}(h) = \int_{N} (\eta - h^*\eta) \wedge \lambda
\]
for $h \in \G(N)$.
This $F_{\lambda}$ gives rise to the following formula, which is just \refthm{transgression_flux}:

\begin{thm}
    The following hold:
    \begin{align*}
        &\Flux_{\lambda}(g) = F_{\lambda}(g) \\
        &\delta F_{\lambda} (h_1, h_2) = \chi(p(h_1), p(h_2))
    \end{align*}
    for $g \in \Gr(N)$ and $h_1, h_2 \in \G(N)$.
    Here, $p: \G(N)\to \Diff_0(S^1)$ is the boundary restriction map, and  $\chi$ is the cocycle, defined in \eqref{eq:Euler_cocycle}, for $\phi(\theta)d\theta = i^*\eta$ and $\psi(\theta)d\theta = i^*\lambda$.
    In particular, the $\lambda$-flux homomorphism transgresses to $A_{\omega}B_{\lambda}\cdot \mathrm{eu}$ with respect to the exact sequence 
    \[
        1 \to \Gr(N) \xrightarrow{i} \G(N) \xrightarrow{p} \Diff_0(S^1) \to 1.
    \]
\end{thm}

\begin{proof}
    The first equality is immediate from the definitions of $\Flux_{\lambda}$ and $F_{\lambda}$.
    Let $h_1,h_2$ be elements of $\G(N)$.
    Then, we have
    \[
        \delta F_{\lambda}(h_1,h_2) = \int_N \big(\eta - h_1^*\eta - h_2^*(\eta - h_1^* \eta) \big) \wedge \lambda = \int_N (\eta - h_1^*\eta) \wedge (\lambda - (h_2^{-1})^*\lambda).
    \]
    Since $\lambda$ is closed and $h_2$ is isotopic to the identity, there exists $\beta_{h_2} \in \Omega^0(N)$ such that $d\beta_{h_2} = \lambda - (h_2^{-1})^*\lambda$.
    Then, the Stokes formula implies
    \[
        \delta F_{\lambda}(h_1, h_2) = -\int_{\partial N} (\eta - h_1^*\eta) \wedge \beta_{h_2} = -\int_{S^1} (i^*\eta - i^*h_1^*\eta) \wedge i^*\beta_{h_2}.
    \]
    Set $\phi(\theta)d\theta = i^* \eta$, $\psi(\theta)d\theta = i^*\lambda$, and $\gamma_j = p(h_j)$ for $j = 1,2$.
    Then, we have
    \[
        \delta F_{\lambda} (h_1, h_2) = -\int_{S^1} \big(\phi(\theta)d\theta - \gamma_1^*(\phi(\theta)d\theta)\big) \wedge i^*\beta_{h_2}.
    \]
    Recall from \refsec{pre_eulerclass} that $\alpha_{\gamma_1}$ is a $0$-form on $S^1$ satisfying $d\alpha_{\gamma_1} = \phi(\theta)d\theta - \gamma_1^*(\phi(\theta)d\theta)$.
    Since
    \[
        d(\alpha_{\gamma_1} \wedge i^*\beta_{h_2}) = \big(\phi(\theta)d\theta - \gamma_1^*(\phi(\theta)d\theta)\big) \wedge i^*\beta_{h_2} + \alpha_{\gamma_1} \wedge \big(\psi(\theta)d\theta - (\gamma_2^{-1})^*(\psi(\theta)d\theta)\big),
    \]
    the Stokes formula implies
    \[
        \delta F_{\lambda} (h_1, h_2) = \int_{S^1} \alpha_{\gamma_1} \wedge \big(\psi(\theta)d\theta - (\gamma_2^{-1})^*(\psi(\theta)d\theta)\big) = \int_{S^1} (\alpha_{\gamma_1} - \gamma_2^*\alpha_{\gamma_1}) \psi(\theta)d\theta.
    \]
    The last term is just $\chi(\gamma_1, \gamma_2)$. This completes the proof.
\end{proof}

\section{Swept-area and $\Fluxn$}\label{Sec:swept_area}

In this section, we characterize the $\lambda$-flux homeomorphism $\Fluxnl$ in terms of the swept areas of arcs associated with $\lambda$. 
First, we define the swept area in surfaces, possibly non-orientable, as follows:

\begin{defn}\label{Defn:sweptArea}
    Let $F$ be a compact, connected surface with boundary, equipped with an area density  form $\omega$.
    Let $\varphi \in \Gn(F)$ and $\gamma:[0,1]\to F$ a proper embedding, that is not boundary-parallel and is oriented by $\gamma^\flat$. 
    Choose a smooth homotopy $h:[0,1]\times [0,1]\to F$ from $\gamma$ to $\varphi\circ \gamma$, oriented by $h^\flat$ with $h^\flat {\restriction_{\{0\}\times [0,1]}}=\gamma^\flat$.
    Then, the \emph{swept-area of $\varphi$ with respect to $\gamma$} is  
    $$\cO_\gamma(\varphi)=-\int_{[0,1]\times [0,1]}h^*\omega.$$
\end{defn}
The following proposition is well known when $F$ is orientable:
\begin{prop}
    The swept-area $\cO_\gamma(\varphi)$ does not depend on the choice of $h$. 
    Moreover, for any proper arc $\sigma$ with an orientation $\sigma^\flat$, isotopic to $\gamma$ with $\gamma^\flat$, that is, there is a smooth homotopy $h:[0,1]\times [0,1]\to F$ from $\gamma$ to $\sigma$, oriented by $h^\flat$ with $h^\flat {\restriction_{\{0\}\times [0,1]}}=\gamma^\flat$ and $h^\flat {\restriction_{\{1\}\times [0,1]}}=\sigma^\flat$, then  $\cO_\gamma=\cO_\sigma$.
\end{prop}

To see the following duality, we can follow the computation of \cite[page~67]{BottTu82} since a tubular neighborhood of any properly embedded arc in a compact surface is homeomorphic to a trivial line bundle over the closed interval: 
\begin{prop}[Poincar\'e dual of an oriented arc]\label{Prop:PD}
    Let $F$ be a compact, connected surface with non-empty boundary and let $\gamma:[0,1]\to F$ be a proper embedding, oriented with $\gamma^\flat$.
    Write $A=\gamma([0,1])$.
    Let $N$ be the normal bundle of $A$, which is the quotient of $TF|_A$ by $TA$, and $j:N\hookrightarrow F$ an embedding onto a tubular neighborhood of $A$, that is the identity on the zero-section of $N$.
    Then, the pushforward $j_*\Phi$ of the Thom class $\Phi$ of $N$ satisfies that for any $\mu\in \Omega^1(F, \partial F;L_F)$,  
    \[
    \int_F \mu \wedge j^* \Phi =\pm\int_{[0,1]} \gamma^*\mu. 
    \]
    where $j_*\Phi$ is the extension of the pushforward of $\Phi$ by $0$.
    Here, the sign depends on $\gamma^\flat$  and  $j$.  
\end{prop}
When the sign is positive, we say that the triple $(\gamma,\gamma^\flat,j)$ is  \emph{well-arranged}. For an associated triple $(\gamma,\gamma^b,j)$, we denote the $1$-form $j^*\Phi$ by $\lambda_\gamma$ and call it the \emph{Poincar\'e dual} of $(\gamma,\gamma^\flat)$ with respect to $j$.

Now, we are ready to show the following characterization:
\begin{lem}[Swept-Area Characterization of $\Fluxnl$]\label{Lem:sweptArea}
    Let $F$ be a compact, connected surface with boundary, equipped with an area density  form $\omega=d\eta$ and $\varphi \in \Gn(F)$.
    Assume that $(\gamma,\gamma^\flat,j)$ is a well arranged triple, given as in \refprop{PD}, and $\lambda_\gamma$ is the Poincar\'e dual of $(\gamma,\gamma^\flat,j)$.
     Then, 
    \[
    \Fluxnl[\lambda_\gamma](\varphi)=\cO_\gamma(\varphi).
    \]
\end{lem}
\begin{proof}
    Choose a smooth isotopy $\varphi_t\in \Gn(F)$ from the identity to $\varphi$.
    Let $X_t$ be the time-dependent vector field generating $\varphi_t$, namely,
    $$\frac{d}{dt}\varphi_t =X_t\circ \varphi_t \text{ and }\varphi_0=id.$$
    Define a smooth homotopy $h:[0,1]\times [0,1]\to F$ from $\gamma$ to $\varphi\circ \gamma$ as 
    $h(s,t)=\varphi_s\circ \gamma(t)$, oriented by $h^\flat$ with $h^\flat {\restriction_{\{0\}\times [0,1]}}=\gamma^\flat$.
     By differentiating, 
     $$\partial_s h=d\varphi_s(\gamma)X_s(\gamma(t)) \text{ and }\partial_t h=d\varphi_s(\gamma) \dot{\gamma}(t)$$
     Since $\varphi_s^*\omega=\omega$,
     $$\cO_\gamma(\varphi)=-\int_0^1\int_0^1 \omega(X_s(\gamma(t)),\dot{\gamma}(t))dsdt.$$

     On the other hand, since $\omega=d\eta$ and 
     $$\frac{d}{ds}[\varphi_s^*\eta]=[\varphi_s^*\{i(X_s)d\eta+d(i(X_s)\eta)\}]=[\varphi_s^*i(X_s)\omega]=[i(X_s)\omega],$$
     we have that
     $$\Fluxn(\varphi)=[\eta-\varphi^*\eta]=\int_1^0\frac{d}{ds}[\varphi_s^*\eta]ds=-\int_0^1[i(X_s)\omega]ds.$$
    Thus, since $(\gamma,\gamma^\flat,j)$ is well arranged, it follows from \refprop{PD} that 
    \[
    \cO_\gamma(\varphi)=\int_{[0,1]}  \gamma^*\Fluxn(\varphi)=\int_F \Fluxn(\varphi) \wedge \lambda_\gamma=\Fluxnl[\lambda_\gamma](\varphi).
    \]
\end{proof}

\begin{figure}[ht]
        \centering
\includegraphics[width=0.5\linewidth]{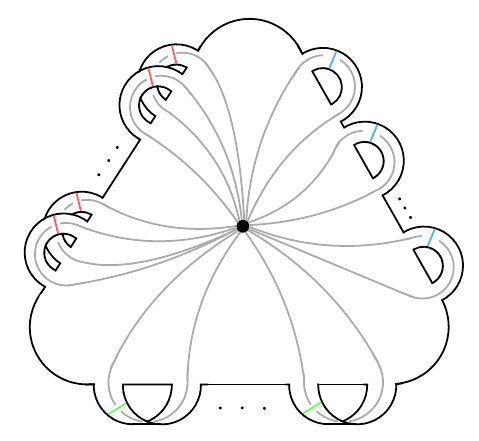}
        \caption{A compact surface with boundary}
        \label{Fig:surface}
\end{figure}

Based on \reflem{sweptArea}, we can see that $\Fluxn$ is surjective. Moreover, it provides a characterization of $\ker(\Fluxn)$.
To see this, we recall the classification theorem of compact surfaces.

By the classification of compact surfaces, any compact surface $F$ with non-empty boundary is obtained as the connected sum of disks $D^2$, tori $T^2$, and real projective planes $\RP^2$. 
Hence, $F$ can be factorized as a connected sum of copies of  $D^2$, $T^2$ and $\RP^2$.
In particular, since  $F$ has a boundary, there is at least one $D^2$ factor. 

A compact surface with non-empty boundary can be represented as a disk with several bands attached.
In \reffig{surface}, several bands are attached along the sides of the central triangular disk.
On the left side, each pair of bands is attached alternately; each pair represents the torus factor.
On the right side, the bands corresponding to the $D^2$-factors are attached consecutively.
Along the bottom, the half‑twisted bands corresponding to the $\RP^2$-factors are also attached consecutively.

Note that if the surface is orientable, then there is no $\RP^2$-factor and we do not need the bands in the bottom side.
Also, when the surface is non-orientable, that is, the factorization of the surface has at least one $\RP^2$-factors, we do not need the bands in the left side by the relation that $\RP^2 \# \RP^2\#\RP^2\cong \RP^2\# T^2$.

By taking a proper embedded arc in each band, we can see that there is a finite collection of properly embedded arcs $\{\alpha_i\}_{i=1}^n$ such that the complement of the union of arcs is a connected contractible subset. 
See the colored arcs in \reffig{surface}.

\begin{prop}\label{Prop:cutSystem}
    Let $F$ be a connected, compact surface with non-empty boundary, equipped with an area density  form $\omega$.
    Then, there is a finite collection of well-arranged triples $\{(\gamma_i,\gamma^\flat_i,j_i)\}_{i=1}^n$ satisfying the followings:
    \begin{itemize}
        \item the tubular neighborhoods $T_i=\im(j_i)$ are pairwise disjoint;
        \item the complement of the union of $\im(\gamma_i)$ is contractible;
        \item the Poincar\'e duals $\lambda_i$ of $(\gamma_i,\gamma^\flat_i,j_i)$ form a basis of $H^1(F)$.
     \end{itemize}
     We call such a collection a \emph{cut system} of $F$.
\end{prop}

\begin{prop}
    Let $F$ be a connected, compact surface with non-empty boundary, equipped with an area density  form $\omega=d\eta$.
    Then, $\Fluxn(F)$ is surjective.
\end{prop}
\begin{proof}
    By \refprop{cutSystem}, we can take a cut system $\{(\gamma_i,\gamma^\flat_i,j_i)\}_{i=1}^n$.
     Also, we use the notations $T_i,\lambda_i$ in \refprop{cutSystem}.
   Once we show the surjectivity of $\Fluxnl[\lambda_i]$, the surjectivity of $\Fluxn$ follows.

   For each $i\in \{1,2,\cdots, n\}$, we can take a smooth simple closed curve $\beta_i$ such that $\beta_i$ does not intersect $T_j,i\neq j$ and $\beta_i$ intersect transversely $\gamma_i$ at a unique point (as the gray curve in \reffig{surface}). 
   For a sufficiently small tubular (closed) neighborhood $N_i$ of $\beta_i$, we may assume that $N_i$ does not intersect $T_j, i\neq j$  and $N_i\cap \im(\gamma_i)$ is a properly embedded arc $d_i$ in $N_i$. 
   Note that $N_i$ is either a closed annulus or M\"obius band. 

   Observe that any Poincar\'e dual $\lambda$ associated with $d_i$ in $N_i$ is a generator of $H^1(N_i)$ since $\dim H^1(N_i)=1$. 
   Due to \reflem{sweptArea} and Moser's theorem (\refthm{Moser}), it is enough to show that when $N$ is either a closed annulus or M\"obius band, the flux homomoprhism $\Fluxn:\Gn(N)\to \RR$ is surjective.
   In the case of annulus, it is well-known. 
   Here, we discuss the case where $N$ is a closed M\"obius band $M$.

   For the simplicity, we assume that $M$ is the quotient space of $\widetilde{M}=\RR\times I$ by the deck transformation defined as $\tau(x,y)=(x+1,-y)$ where $I=[-1/2,1/2]$.
   Say that $\pi:\widetilde{M}\to M$ is the quotient map.
   Also, we use $\eta=-xdy\otimes e$ for some global constant section $e$ of $L_{\widetilde{M}}$ such that $e=1$ or $e=-1$ in any local trivialization. 
   Note that $d\eta$ is the standard area density  form $dx\wedge dy \otimes e$ on $\widetilde{M}$.
   Moreover, $\eta$ and $d\eta$ induce well-defined twisted forms in $M$.

   Now, we define a one-parameter family $p_t$ of $\omega$-preserving diffeomorphisms in $\Gn(M)$ such that $\Fluxn(\{p_t:t\in \RR \})=H^1(M)\cong \RR$.
   Let $b:I\to \RR$ be a bump function on $I$, satisfying
\begin{itemize}
    \item $b$ is an even function, that is,  $b(y)=b(-y)$;
    \item the support of $b$ is $[-1/4,1/4]$;
    \item $b=1$ on $[-1/8,1/8]$.  
\end{itemize}

Define a vector field $\widetilde{X}$ on $\widetilde{M}$ as 
$\widetilde{X}(x,y)=b(y)\partial/\partial x$ and then this induces a smooth vector field $X$ on $M$. We put $\tilde{p}_t$ and $p_t$ as the $1$-parameter family of $\widetilde{X}$ and $X$, respectively. 
Then, $\pi\circ  \tilde{p}_t =p_t$.
By the construction,  
$$\tilde{p}_t(x,y)=(x+t \cdot b(y),y)$$
and $\tilde{p}_t$ and  $\tilde{p}_t$ are $\omega$-preserving diffeomorphisms.
Moreover, 
$$\eta-\tilde{p}_t^*\eta=-ydx \otimes e + yd(x+tb(y)) \otimes e
=t\cdot yb'(y)dy\otimes e.$$
By abusing the notation, we have that 
\[\Fluxnl[dx](p_t)=-\int_M t\cdot y b'(y)dx\wedge dy\otimes e=-t\int_M y b'(y) \omega.\] 
Note that $b':I\to \RR$ is an odd function  supported on $[-1/4,1/4]\setminus (-1/8,1/8)$ and $\int_{0}^{1/2} b'(y)dy\neq 0$. 
Hence, $\int_M y b'(y) \omega\neq 0$.
Therefore, $\Fluxnl[dx](\{p_t:t\in \RR\})=\RR$. 
Since $[dx]$ is a generator of $H^1(M)$, this implies the desired surjectivity.   \qedhere
\end{proof}

In the similar way, we can charaterize an element in $\ker(\Fluxn)$ as follows:
\begin{lem}\label{Lem:zeroArea}
    Let $F$ be a connected, compact surface with non-empty boundary, equipped with an area density  form $\omega$ and $\{(\gamma_i,\gamma^\flat_i,j_i)\}_{i=1}^n$ a cut system of $F$.
    Let $\lambda_i$ be the Poincar\'e duals  of $(\gamma_i,\gamma^\flat_i,j_i)$.
    Then, for any $g\in \Gn(F)$, $g\in \ker(\Fluxn)$ if and only if $\Fluxnl[\lambda_i](g)=0$ for all $i$.
\end{lem}
\begin{proof}
    If $g\in \ker(\Fluxn)$, then it follows from the definition that $\Fluxnl[\lambda_i](g)=0$ for all $i$.
    The other direction follows from the condition that the Poincar\'e duals $\lambda_i$ of $(\gamma_i,\gamma^\flat_i,j_i)$ form a basis of $H^1(F)$.
\end{proof}

\section{Fragmentation Lemma on $\ker(\Fluxn)$}\label{Sec:excision}

The sympletic fragmentation lemma is one of the key ingredients in the proof of the simplicity of the kernel of the Calabi invariant in orientable surfaces (see \cite[page~110]{Banyaga97} for the proof):
\begin{lem}[The symplectic fragmentation lemma]\label{Lem:sympFrag}
Let $\cU=\{U_i\}_{i\in I}$ be an open cover of a connected symplectic manifold $X$ with a symplectic form $\omega$.
If $\varphi$ is a Hamiltonian diffeomorphisms on $X$, then $\varphi$ can be written as 
\[
\varphi=\varphi_1\varphi_2\cdots\varphi_\fN
\]
where each $\varphi_i$ is a Hamiltonian diffeomorphism, supported in some $U_{n(i)}\in \cU, n(i)\in I$.
In particular, if $X$ is not compact and $\Cal(\varphi)=0$, then we can choose that $\Cal_{U_{n(i)}}(\varphi_i)=0$.
\end{lem}

In this section, we show a version of fragmentation lemma for $\ker(\Fluxn)$ (\reflem{frag}).
To do this, we first prove some excision lemmas, following \cite{EntovPolterovichPy} and \cite{Serraille}.

\subsection{Area-preserving excision lemma}
For a homeomorphism $f$ on a topological space, we define the \emph{support} of $f$  as the closure of $\{x\in X: f(x)\neq x\}$ and denote it by $\supp(f)$.

First, we remark the smooth excision lemma:
\begin{lem}[Smooth excision lemma]\label{Lem:smoothExcision}
    Let $R=[0,\ell]\times [-w,w]$ be a rectangle and $R_1\subset R_2\subset R$ be two smaller rectangles of the form $R_i=[0,\ell]\times [-w_i,w_i],i\in\{1,2\}$ with $0<w_1<w_2<w$. Assume that $R$ is equipped with an area density  form $\nu$.
    Let $F$ be a connected, compact surface with non-empty boundary, equipped with an area density  form $\omega$.  Suppose that there is a smooth area-preserving  embedding $\delta:R\to F$ (that is, $\delta^*\omega=\nu$ ) such that $\delta(R)\cap \partial F=\delta(\{0,\ell\}\times [-w,w])$. Let $U$ be an open subset of $\delta(R)$ such that $\delta(R_1)\subset U \subset \delta(R_{2})$. If $\varphi\in \Gn(F)$ and there is an isotopy $\varphi_t\in \Gn(F)$ such that 
    \begin{itemize}
        \item $\varphi_0=id$ and $\varphi_1=\varphi$;
        \item $\varphi_t(U)\subset \delta(R_{2})$ for all $t\in [0,1]$,
    \end{itemize}
    then, there exists $\psi$ and an isotopy $\psi_t$ in the identity component $\Diff_0(F,\near \partial F)$ of the group of diffeomorphisms  fixing some neighborhoods of $\partial F$ such that $\psi_0=id$, $\psi_1=\psi$, $\supp(\psi_t)\subset \delta(R)$ and $\psi_t=\varphi_t$ on $U$.
\end{lem}
\begin{proof}
    Let $X_t$ be the time-dependent vector field generating $\varphi_t$. Then, we take a smooth function $\lambda:F\to \RR$ such that $\lambda(F)\subset [0,1]$, $\closure{F\setminus R} \subset \lambda^{-1}(0)$ and $R_{2}\subset \lambda^{-1}(1)$.
    The isotopy $\psi_t$, generated by the time-dependent vector field $\lambda X_t$, satisfies the following properties:
    \begin{itemize}
        \item $\psi_t\in \Diff_0(F,\near \partial F)$;
        \item $\psi_t$ are the identity outside of $\delta(R)$;
        \item $\psi_t=\varphi_t$ on $U$ for all $t\in [0,1]$. 
    \end{itemize}
    The third property follows from the condition that $R_{2}\subset \lambda^{-1}(1)$ and $\varphi_t(U)\subset \delta(R_{2})$ for all $t\in [0,1]$.
    Then, $\psi_1$ is the desired diffeomorphism $\psi$.
\end{proof}

Then we promote the diffeomorphism given by \reflem{smoothExcision} to an area-preserving diffeomorphism:

\begin{lem}[Area-preserving excision lemma]\label{Lem:densityExcision}
    Let $R=[0,\ell]\times [-w,w]$ be a rectangle and $R_1\subset R_2\subset R_3 \subset R$ three smaller rectangles of the form $R_i=[0,\ell]\times [-w_i,w_i],i\in\{1,2,3\}$ with $0<w_1<w_2<w_3<w$. Assume that $R$ is equipped with an area density  form $\nu$.
    Let $F$ be a connected, compact surface with non-empty boundary, equipped with an area density  form $\omega$.  
    Suppose that there is a smooth area-preserving embedding $\delta:R\to F$ such that $\delta^*\omega=\nu$ and   $\delta(R)\cap \partial F =\delta(\{0,\ell\}\times [-w,w])$. Let $U$ be an open subset of $\delta(R)$ such that $\delta(R_1)\subset U \subset \delta(R_{2})$.
    If $\varphi\in \ker(\Fluxn)$ and there is an isotopy $\varphi_t\in \ker(\Fluxn)$ such that 
    \begin{itemize}
        \item $\varphi_0=id$ and $\varphi_1=\varphi$;
        \item $\varphi_t(U)\subset \delta(R_{2})$ for all $t\in [0,1]$,
    \end{itemize}
    then, there exists $\psi \in \Gn(F)$ such that $\supp(\psi)\subset \delta(R)$ and $\psi=\varphi$ on $\delta(R_1)$.
\end{lem}
\begin{proof}
    Let $\psi_t$ be the isotopy given by \reflem{smoothExcision} such that 
    $\supp(\psi_t)\subset \interior(\delta(R_3))$ and $\psi_t=\varphi_t$ on $U$. We consider $\Omega=\psi_1^*\omega$.
Say that $$R_-=\delta([0,\ell]\times [-w,-w_1]) \text{ and }R_+=\delta([0,\ell]\times [w_1,w])$$
    and also,
    $$d_-=\delta([0,\ell]\times -w_1) \text{ and }d_+=\delta([0,\ell]\times w_1).$$
  By \reflem{sweptArea} and \reflem{zeroArea}, for any well-arranged triple $(\gamma,\gamma^\flat,j)$ with $\im(\gamma)=d_\alpha,\alpha\in \{+,-\}$, we have that  $\cO_\gamma(\varphi)=0$ and so the (signed) area swept out by each of $d_+$ and $d_-$ under the isotopy $\varphi_t$ from $0$ to $\varphi$ is zero. 
  Therefore, since $\supp(\psi_t)\subset \delta(R_3)$, $\psi_t=\varphi_t$ on $U$ and $d_+,d_-\subset U$,
    $$\int_{R_+}\Omega=\int_{R_+}\omega \text{ and }\int_{R_-}\Omega=\int_{R_-}\omega.$$ 
    By the Moser's theorem for manifolds with corners (\cite[7~Theorem]{Bruveris}), for each $\alpha\in \{+,-\}$, there is  $h_\alpha \in \Diff_0(R_\alpha,\partial R_\alpha)$ such that 
     \begin{itemize}
         \item $h_\alpha^*\Omega=\omega$ on $R_\alpha$;
          \item $h_\alpha$ is the identity on  the boundary of $R_\alpha$.
     \end{itemize}
    Moreover, since $\supp(\psi_1) \subset \interior(\delta(R_3))$ and $\psi_1 \in \Gn(F)$, we have $\Omega = \omega$ on a neighborhood of $\partial R_\alpha \setminus d_\alpha$. Since $\psi = \varphi$ on $U \supset d_\alpha$ and $\varphi^* \omega = \omega$, it follows that $\Omega = \omega$ on a neighborhood of $d_\alpha$ as well.
    Therefore, $\Omega = \omega$ on a neighborhood of $\partial R_\alpha$ and so we can take $h_\alpha$, that is the identity on neighborhood the boundary of $\partial R_\alpha$.
    This is not stated explicitly in \cite[7~Theorem]{Bruveris}, but it is implied by its proof.
     
    Then, we can define $h\in \Diff_0(F,\near 
     \partial F)$ as follows:
     \begin{align*}
             h(x)=
              \begin{dcases}
                     h_+(x) & \text{if } x\in R_+,\\
                h_-(x) & \text{if } x\in R_-,\\
                x & \text{, otherwise.}\\
              \end{dcases}
     \end{align*}
    It follows from the first property of $h_\alpha$ that
    $h^*\Omega=\omega$ and $(\psi_1\circ h)^* \omega=\omega$.
    
    Finally, we claim that $\psi=\psi_1\circ h$ is the desired diffeomorphism.
    It follows from the construction that $\supp(\psi)\subset \delta(R)$ and $\psi=\varphi$ on $\delta(R_1)$.
    Therefore, we only need to show that $\psi\in \Gn(F)$.
    Since $\psi$ is compactly supported on $\delta(R)$ and on a closed embedded disk $\cD$, there is a smooth isotopy from $\psi{\restriction_\cD}$ to $id$ in $\Gn(\cD)$.
    Such an isotopy can be freely extended to a smooth isotopy in $\Gn(F)$ by extending the maps as the identity outside $\cD$.
    This provides a smooth isotopy from $\psi$ to $id$ in $\Gn(F)$. Thus, we are done.\qedhere
\end{proof}

\subsection{Fragmentation lemma}
To prove the fragementation lemma for an element $\varphi$ in $\ker(\Fluxn)$, we make use of \reflem{densityExcision}. 
Hence, we need a smooth isotopy in $\ker(\Fluxn)$ connecting $\varphi$ with $id$. 
The existence of such a isotopy is guaranteed by the following fact: 
\begin{prop}\label{Prop:arcConnected}
    $\ker(\Fluxn)$ is smoothly arcwise connected.
\end{prop}
\begin{proof}
    It is obtained by repeating the argument of the proof of \cite[Proposition~4.2.1]{Banyaga97}
\end{proof}

\begin{lem}[Fragmentation lemma]\label{Lem:frag}
    Let $F$ be a connected, compact surface with non-empty boundary, equipped with an area density  form $\omega$. 
    Then, for each $h\in \ker(\Fluxn)$, there exist finitely many elements $h_1,h_2, \cdots, h_N$ in $\Gn(F)$ such that $h=h_1h_2\cdots h_N$ and each $h_i$ is  compactly supported in an open disk. 
\end{lem}
\begin{proof}
Take a cut system $\{(\gamma_i,\gamma^b_i, j_i)\}_{i=1}^n$ and write $T_i=\im(j_i)$.
Since $\closure{T_i}$ is a closed rectangle, by Moser's theorem for manifolds with corners (\cite[7~Theorem]{Bruveris}), we can take a smooth area-preserving embedding $\delta_i:R^i\to \closure{T_i}$, as in \reflem{densityExcision}, where $R_i=[0,\ell_i]\times [-w_i,w_i]$. 
Then, by taking the image of some restriction of $\delta_i$ on the second coordinate, we can take a smaller rectangle $S_i$ such that $S_i\subset T_i$. 
Choose contractible open subsets $U_i$ of $F$ such that $S_i\subset U_i \subset \closure{U_i}\subset T_i$.

By \refprop{arcConnected}, there is an isotopy $h_t$ in $\ker(\Fluxn)$ from $id$ to $h$.
For a sufficiently large $\fM\in \NN$, 
\[
h_{k/\fM}^{-1}\circ h_{(k+t)/\fM}(\closure{U_i})\subset T_i
\]
for all $(k,t)\in \{0,1,\cdots, \fM-1\}\times[0,1]$ and for all $i\in \{1,2,\cdots, n\}$.
Fix such an $\fM$ and write \[f_k=h_{k/\fM}^{-1}\circ h_{(k+1)/\fM}, \ k\in \{0,1,\cdots, \fM-1\} \text{ and } f_{k,t}=h_{k/\fM}^{-1}\circ h_{(k+t)/\fM},\ t\in [0,1].\]
Note that $f_{k,t}$ is an isotopy from $id$ to $f_k$ in $\ker(\Fluxn)$.

Since \(h=f_0f_1\cdots f_{\fM-1}\), it is enough to show the fragmentation lemma for the case of $f_k$.
By the choice of $\fM$ and \reflem{densityExcision}, there are $\psi_{k,i}\in \Gn(F)$ such that $\supp(\psi_{k,i})\subset T_i$ and $\psi_{k,i}=f_k$ on $S_i$.

Now, we set $\psi_{k,0}=\psi_{k,1}^{-1}\cdots \psi_{k,n}^{-1}f_k$.
Since  $\supp(\psi_{k,i})\cap \supp(\psi_{k,j})=\varnothing$ and $\psi_{k,i}=f_k$ on $S_i$, 
$\supp(\psi_0)$ does not intersect $\im(\gamma_i)$, and so, it is supported in an open disk.
Thus, \[f_k=\psi_{k,n}\cdots \psi_{k,1}\psi_{k,0}\] and we are done.
\end{proof}

We end this section with the following which follows from \reflem{zeroArea}:

\begin{lem}\label{Lem:zeroFluxn}
    Let $F$ be a connected, compact surface with non-empty boundary, equipped with an area density  form $\omega$. 
    If an element $h$ in $\Gn(F)$ is supported on an open disk, then $\Fluxn(h)=0$.
\end{lem}

\section{Simplicity of $\ker(\Fluxn)$}
In this section, we discuss the simplicity of $\ker(\Fluxn)$ of a connected, compact, non-orientable surface $N$ with non-empty boundary. 
The key idea of the proof is the \emph{Cell Division Trick},  which is a generic phenomenon on non-orientable surfaces.

\begin{lem}[Cell Division Trick]\label{Lem:cellDivision}
Let $M$ be a M\"obius band, equipped with an area density  form $\omega$. 
If an element $h$ in $\ker(\Fluxn)$ is compactly supported in an open disk, then there are open disks $U,V$ and $u,v\in \Gn(M)$ such that $u$ and $v$ are compactly supported in $U$ and $V$, respectively, $\Cal_{U}(u)=\Cal_{V}(v)=0$, and $h=u v$. 
\end{lem}
\begin{proof}
Since $h$ is compactly supported in an open disk, we can take two open  subsets $U,V$ satisfying the followings:
\begin{itemize}
    \item $U$ and $V$ are open disks bounded by Jordan curves;
    \item $U\cap V$ is the disjoint union of open disks, $A$ and $B$;
    \item $U\cup V$ is homeomorphic to the open M\"obius band;
    \item $\supp(h)\subset U \setminus \closure{V}$.
\end{itemize}
Then, we fix standard locally constant sections $e_U$ and $e_V$ over $U$ and $V$, respectively, such that $e_U=e_V$ in $A$ and $e_U=-e_V$ in $B$.

If the local Calabi invariant of $h$ on $U$ with respect to $e_U$ vanishes, namely, $\Cal_U(h)=0$, then we are done by putting $u=h$ and $v=id$.

Otherwise, $\Cal_U(h)\neq 0$.
Write $c=\Cal_U(h)/2$ and set locally constant sections $e_A$ and $e_B$ as the restriction of $e_U$ on $A$ and $B$, respectively.
By the surjectivity of Calabi invariant, there are $g_A$ and $g_B$ in $\Gn(A)$ and $\Gn(B)$, repsectively, such that $\Cal_A(g_A)=\Cal_B(g_B)=-c$ with respect to $e_A$ and $e_B$, respectively.
Since $e_A$ and $e_B$ are the restriction of $e_U$, we have
\[
\Cal_U(g_A)=\Cal_A(g_A)=-c \text{ and }\Cal_U(g_B)=\Cal_B(g_B)=-c
\]
with respect to $e_U$,
and 
since $e_U=e_V$ on $A$ and  $e_U=-e_V$ on $B$,
\[
\Cal_V(g_A)=\Cal_A(g_A)=-c \text{ and }\Cal_V(g_B)=-\Cal_B(g_B)=c
\]
with respect to $e_V$.
This implies that $\Cal_U(hg_Ag_B)=0$ and $\Cal_V(g_B^{-1}g_A^{-1})=0$.
Thus, $u=hg_Ag_B$ and $v=g_B^{-1}g_A^{-1}$ are the desired elements.   
\end{proof}
\begin{rmk}
    The name \emph{Cell division trick} comes from the technique of using two small elements $g_A$ and $g_B$ to split the twisting effect of $h$.
\end{rmk}

Now, we can promote \reflem{frag} in the following form:
\begin{lem}[Strong fragmentation lemma]\label{Lem:strongFrag}
Let $N$ be a compact, non-orientable surface with non-empty boundary, equipped with an area density  form $\omega$. Then,
for any $h\in \ker(\Fluxn)$, there are $h_1,\cdots, h_N$ in $\ker(\Fluxn)$ such that $h=h_1h_2\cdots h_N$, $h_i$ are supported in open disks $B_i$ and $\Cal_{B_i}(h_i)=0$.
\end{lem}
\begin{proof}
    By \reflem{frag}, there exist finitely many elements $g_1,g_2,\cdots, g_N$ in $\Gn(F)$ such that $h=g_1g_2\cdots g_N$ and each $g_i$ is compactly supported in an open disk $U_i$. 
    Moreover, by \reflem{zeroFluxn}, $\Fluxn(g_i)=0$.
    Since $N$ is non-orientable, for each $i$, we can take an embedded closed M\"obius band $M_i$, that  contains $U_i$ (see \reffig{surface}).
    Thus, the desired result follows from \reflem{cellDivision} and \reflem{zeroFluxn}. 
\end{proof}

Under \reflem{cellDivision} and \reflem{strongFrag}, the simplicity of $\ker(\Fluxn)$ for non-orientable surfaces follows from Thurston's trick, explained in \cite[Section 2.1]{Banyaga97}:
%For the convenience of the reader, we provide the proof.
\begin{restate}{Theorem}{Thm:simplicity}[Simplicity of $\ker(\Fluxn)$]
Let $N$ be a compact, non-orientable surface $N$ with non-empty boundary, equipped with an area density  form $\omega$. Then, $\ker(\Fluxn)$ is simple.
\end{restate}
\begin{proof}
    Write $G=\ker(\Fluxn)$. 
    It is enough to show that for any $\varphi \in G \setminus \{id\}$, the normalizer $N_G(\varphi)$ is $G$.
    Fix $\varphi\in G\setminus \{id\}$ and take a small open disk $B$ in $N$ and $g\in G$ satisfying the followings:
    \begin{itemize}
        \item $B\cap \varphi(B)=\varnothing$;
        \item $\varphi(B)\cap g(\varphi(B))=\varnothing$;
        \item $g$ is the identity on $B$.
    \end{itemize} 
    Note that $B\subset \closure{B}\subset \supp(\varphi)\subset \interior(N)$.
    \begin{claim}\label{Clm:commutator}
        For any $u,v\in \ker(\Cal_B)$, $[u,v]\in N_G(\varphi)$.
    \end{claim}
    \begin{proof}[Proof of claim]
        First, note that  \[[u,\varphi]=u\varphi u^{-1}\cdot \varphi^{-1}\in N_G(\varphi).\] 
        Observe that $[u,\varphi]$ is decomposed into $u$ and $\varphi u^{-1} \varphi^{-1}$, which are supported on $B$ and $\varphi(B)$, respectively,
        
        On the other hand,
        \[
        [v,g\varphi g^{-1}]=(vg)\varphi (vg)^{-1}\cdot g\varphi^{-1} g^{-1}\in N_G(\varphi).
        \]
        Moreover, $v=g^{-1}vg$ since $v$ is supported on $B$ and $g$ is the identity on $B$.
        Therefore, 
        \[ 
        v^{-1}[v, g\varphi g^{-1}]=g\varphi g^{-1} \cdot v^{-1} \cdot g\varphi^{-1} g^{-1}=
         (g\varphi) \cdot v^{-1} \cdot (g\varphi)^{-1}
        \]
         and $v^{-1}[v, g\varphi g^{-1}]$ is supported on $g\varphi(B)$.
         Therefore, 
         $[v,g\varphi g^{-1}]$ is decomposed into $v$ and $v^{-1}[v,g\varphi g^{-1}]$ which are supported on $B$ and $g\varphi(B)$. 

         Since $B$, $\varphi(B)$ and $g\varphi(B)$ are pairwise disjoint,
         it follows from the above decompositions that 
         \[
         [u,v]=[[u,\varphi],[v,g\varphi g^{-1}]]\in N_G(\varphi).
         \]
    \end{proof}

Now, we end up the proof, showing that for any $h\in G$, $h\in N_G(\varphi)$.
Fix a non-trivial element $h\in G$.
By \reflem{strongFrag}, there are $h_1,\dots, h_N$ in $G$ such that $h=h_1h_2\cdots h_N$, $h_i$ is supported on an open disk $V_i$ and $\Cal_{V_i}(h_i)=0$.
Once we show $h_i\in N_G(\varphi)$ for all $i\in \{1,2,\cdots, \fN\}$, it follows that $h=h_1h_2\cdots h_\fN\in N_G(\varphi)$.
Hence, it is enough to show that if an element $h$ in $G$ is supported on an open disk $D$ and $\Cal_D(h)=0$, then $h\in N_G(\varphi)$.
By shrinking $D$ if necessary, we can assume that $D$ is bounded by a Jordan curve, contained in $\interior(N)$.

\begin{claim}[Transitivity]\label{Clm:transitivity}
    For any closed subset $A$ of $\interior(N)$, there is an open disk cover $\cU=\{U_i\}_{i\in I}$ of $A$ and associated subset $\{\alpha_i\}_{i\in I}$ of $G$ such that $\closure{U_i}\subset \interior(N)$ and  $\alpha_i(U_i)\subset B$.
\end{claim}
\begin{proof}[Proof of claim]
    It follows from the fact that for any $x\in \interior(N)$, there is a $\alpha \in G$ such that $\alpha(x)\in B$.
\end{proof}
Choose such an open disk cover $\cU=\{U_i\}_{i\in I}$ of $\closure{D}$ and associated subset $\{\alpha_i\}_{i\in I}$ of $G$ for $\supp(h)$.
By applying the symplectic fragmentation lemma (\reflem{sympFrag}) for $D$, there are elements $f_1,f_2,\dots, f_\fM$ in $G$ and  $V_{1},V_{2},\dots, V_{\fM}$ in $\cU$ such that
$h=f_{1}f_{2}\cdots f_{\fM}$, $f_i$ are supported on $V_i$ and $\Cal_{D}(f_i)=\Cal_{V_{i}}(f_{i})=0$.
After relabeling if necessary, we can assume that $V_i=U_i$.
Since $\ker(\Cal_{U_i})$ is perfect, there are $u_{i,j}, v_{i,j}$ in $\ker(\Cal_{U_i})$ such that 
\[
f_i=[u_{i,1}, v_{i,1}][u_{i,2}, v_{i,2}]\cdots [u_{i,\fM_i}, v_{i,\fM_i}].
\]
Since $\alpha_iu_{i,j}\alpha_i^{-1}$ and $\alpha_iv_{i,j}\alpha_i^{-1}$ are in $\ker(\Cal_B)$, 
by \refclm{commutator},
\[\alpha_i[u_{i,j},v_{i,j}]\alpha_i^{-1}=[\alpha_iu_{i,j}\alpha_i^{-1},\alpha_iv_{i,j}\alpha_i^{-1}]\in N_G(\varphi)\]
and so
$[u_{i,j},v_{i,j}]\in N_{G}(\varphi)$.
Thus, $f_i\in N_G(\varphi)$ and 
\[
h=f_1f_2\cdots f_\fM\in N_G(\varphi).
\]
\end{proof}

\appendix
\section{Differential topology with twisted forms}\label{Appendix}

\subsection{Moser's theorem for volume density  forms}\label{Sec:MoserThm}
In \cite{Moser65}, Moser  proved that if $\tau_t$ is a 1-parameter family of volume forms on a connected and compact manifold $\cN$ without boundary, then the condition $\int_\cN \tau_t=\int_\cN \tau_0$, for all $t$, implies the existence of an isotopy $\Phi_t$ of $\cN$ such that $\Phi^*_t\tau_t=\tau_0$.
In fact, since he proved the theorem in terms of odd differential forms, his theorem includes the case of non-orientable manifolds without boundary. After that, Banyaga \cite{Banyaga74} proved the following version of Moser's theorm, which is for an orientable manifold with non-empty boundary.

\begin{thm}
Let $\cN$ be a compact, connected, orientiable, $n$-dimansional manifold with boundary $\partial \cN$ and $\tau_t$ a $1$-parameter family of volume forms. The following conditions are equivalent:
\begin{itemize}
    \item[(i)] $\int_\cN \tau_t=\int_\cN \tau_0$, for all $t$;
    \item[(ii)] There exists a $1$-parameter family $\alpha_t$  of $(n-1)$-forms such that $\partial \tau_t /\partial t=d \alpha_t$ and $\alpha_t(x)=0$ for all $x\in \partial \cN$;
    \item[(iii)] There exists an isotopy $\Phi_t$ on $\cN$ such that 
    $$\Phi^*_t\tau_t=\tau_0, \Phi_0=id \text{ and }\Phi_t|\partial \cN=id.$$
\end{itemize}
\end{thm}

By replacing the ordinary forms with twisted differential forms, every argument in \cite{Banyaga74} can be applicable to non-orientable manifolds.   
Therefore, we have the following version of Moser's theorem.
\begin{thm}\label{Thm:Moser}
Let $\cN$ be a compact, connected $n$-dimansional manifold with boundary $\partial \cN$ and $\omega_t$ a $1$-parameter family of volume density  forms. %(differential forms of degree $n$ of odd kind, everywhere positive). 
The following conditions are equivalent:
\begin{itemize}
    \item[(i)] $\int_\cN \omega_t=\int_\cN \omega_0$, for all $t$;
    \item[(ii)] There exists a $1$-parameter family $\alpha_t$  of $(n-1)$-forms of odd kind such that $\partial \omega_t /\partial t=d \alpha_t$ and $\alpha_t(x)=0$ for all $x\in \partial \cN$;
    \item[(iii)] There exists an isotopy $\Phi_t$ on $\cN$ such that 
    $$\Phi^*_t\omega_t=\omega_0, \Phi_0=id \text{ and }\Phi_t|\partial \cN=id.$$
\end{itemize}
\end{thm}
Also, in \cite{Bruveris}, a version of Moser's theorem was shown for the manifolds with corners, possibly non-orientable, including the case of  \refthm{Moser}. See \cite[7~Theorem]{Bruveris}.

\subsection{Homotopy invariance of twisted de Rham cohomology}
In \cite{deRham84}, de Rham unified the concepts of singular chains and even/odd differential forms in terms of currents. 
Then, he study the homology groups of currents and showed the  the homotopy invariance for homology groups of currents.
See \cite[$\mathsection 18.$ Homology Groups]{deRham84}.
More directly, we can follows the proof of \cite[Corollary~4.1.2.]{BottTu82} with the twisted forms.
Indeed, by following the argument, we can also see that the homotopy invariance holds for relative homology groups $H^*(X,\partial X;L_X)$.

We rephrase the theorem for our purpose as follows:
\begin{prop}[Homotopy invariance of twisted de Rham cohomology]\label{Prop:homotopyInvariance}
    Let $X,Y$ be compact, connected, smooth manifolds, possibly non-orientable, and $F,G$ smooth maps from $X$ to $Y$. If there is a smooth homotopy $H:X\times [0,1]\to Y$ from $F$ to $G$ and $H$ is oriented, then 
    for all $i\geq 0$, the induced homomorphisms
    $F^*, G^*:H^i(Y;L_Y)\to H^i(X;L_X)$ coincide.
    The same statement holds for $ H^*(X,\partial X;L_X)$ and $H^*(Y,\partial Y;L_Y)$ and for $ H^*_c(\mathring{X};L_{\mathring{X}})$ and $H^*_c(\mathring{Y};L_{\mathring{Y}})$.
\end{prop}

\subsection{Relative twisted cohomologies}
We follow the formulation given in \cite{Godbillon71} to define a relative version of twisted cohomology.
First, we recall that formulation from \cite{Godbillon71}.

A \emph{closed $n$-dimensional submanifold} $N$ of a $m$-dimensional manifold $M$ with or without boundary is a closed subset of $M$ such that for any $x\in N$, there is an open neighborhood $U$ of $x$ in $M$ and a diffeomorphism $\varphi$ from $U$ to $\RR^m$ or to the upper-half space $\HH^m$ satisfying the followings:
\begin{enumerate}
    \item if $n<m-1$, then $\varphi(U)=\RR^m$ and $\varphi(U\cap N)=\RR^n$;
    \item if $n=m-1$, then $\varphi(U)$ is either $\RR^m$ or $\HH^m$, and 
    $\varphi(U\cap N)=\RR^{m-1}$.
    \item if $n=m$, then $\varphi(U)$ is  $\RR^m$ or $\HH^m$, and $\varphi(U\cap N)$ is  $\RR^m$ or $\HH^m$.
\end{enumerate}
Note that there exists a unique differentiable structure on such an $N$ such that the inclusion $i:N\hookrightarrow M$ is an embedding.

Then, we assume that each component $C$ of $N$ satisfies the following:
\begin{enumerate}
    \item if $n<m-1$, then $C$ is a closed manifold in the interior $\mathring{M}$;
    \item if $n=m-1$, then $C$ is a closed manifold in $\mathring{M}$ or in $\partial M$;
    \item if $n=m$, then $C$ is a manifold whose boundary components satisfy one of the properties listed above.
\end{enumerate}
We define $\Omega^p(M,N)$ (respectively, $\Omega_c^p(M,N)$) as the space of twisted $p$-forms $\alpha\in \Omega^p(M)$ (respectively, $\mu\in \Omega^p_c(M)$) such that $i^*\mu=0$.
Then, the cohomology $H^*(M,N)$ (respectively, $H^*_c(M,N)$) of the cochain complex $(\Omega^\bullet (M,N),d)$ (respectively, $(\Omega_c^\bullet (M,N),d)$) is well-defined. 

In the similar way, we can define $\Omega^p(M,N;L_M)$ (respectively, $\Omega_c^p(M,N;L_M)$) as the space of $p$-forms $\mu \in \Omega^p(M;L_M)$ (respectively, $\mu \in \Omega^p_c(M;L_M)$) such that $i^*\mu=0$.
Also, the cohomology $H^*(M,N;L_M)$ (respectively, $H^*_c(M,N;L_M)$) of the cochain complex $(\Omega^\bullet (M,N;L_M),d)$ (respectively, $(\Omega_c^\bullet (M,N;L_M),d)$) is well-defined.
We can deduce  the cohomology theories of $H^*(M,N;L_M)$ and  of $H^*_c(M,N;L_M)$ from  the cohomology theories of $H^*(M,N)$ and  of $H^*_c(M,N)$ by tensoring  the coefficient $L_M$.

We remark the following equality:
\begin{prop}\label{Prop:relToCsupp}
    Let $F$ be a compact, connected surface with boundary. Then, 
    \[
    H^i(F,\partial F; L_F)\cong H^i_c(\mathring{F};L_{\mathring{F}})\] 
    for all non-negative integers $i$.
\end{prop}
\begin{proof} 
We have the parameterization of the collar neighborhood of $\partial F$, $i:\partial F \to [0,1)$ such that $i(\partial F\times \{0\})=\partial F$. 
Set $C_t= i(\partial F\times [0,1-t))$ for each $t\in [0,1]$.
Following the computation of \cite[Example~3.34]{Hatcher02}, 
we can see that 
\[
H^i_c(\mathring{F};L_{\mathring{F}})=\varinjlim H^i(F,C_t;L_F)
\]
since any compact subset of $\mathring{F}$ is contained in $F\setminus C_t$ for some $t\in [0,1]$.
Note that $H^i(F,C_s;L_F)\cong H^i(F,C_t;L_F)$ for any $s\neq t \in [0,1]$ since $C_s$ and $C_t$ are  collar neighborhoods of $\partial F$.
Hence, it suffices to show that  $H^i(F,C_t;L_F)\cong H^i(F,\partial F;L_F)$. 
This follows from the long exact sequence for the triple $(F,C_t, \partial F)$ (modifying \cite[5.1~Théorème, XII]{Godbillon71}):
\begin{align*}
\cdots \to H^{i-1}(C_t, \partial F;L_{C_t})\xrightarrow{\delta} H^i(F, C_t;L_{F})\xrightarrow[]{j^*}  H^i(F, \partial F;L_F)\xrightarrow[]{\iota^*} H^i(C_t, \partial F;L_{C_t})\to \cdots 
\end{align*}
since $H^{i}(C_t, \partial F;L_{C_t})=H^{i}(C_t, \partial F)=0$ for all non-negative integers $i$. 
Here, \[j:(F,\partial F)\to (F,C_t) \text{ and } \iota:(C_t, \partial F)\to (F,\partial F)\] are the inclusions, and $\delta$ is a connecting homomorphism such that  for $[\mu]\in H^i(C_t,\partial F;L_{C_t})$ with $d\mu=\nu\in \Omega^{i+1}(F,C_t;L_{F})$, $\delta([\mu])=[\iota^*\nu]$.
\end{proof}

\section*{Acknowledgements}
We would like to thank Takashi Tsuboi, Sangjin Lee, Hongtaek Jung, Mitsuaki Kimura and Erika Kuno for helpful conversations and comments.
In particular, we would like to thank Takashi Tsuboi for sharing the motivation behind his work (\cite{Tsuboi00}) in the conference  \href{https://conferences.cirm-math.fr/3082.html}{``Foliations and Diffeomorphism Groups"} (CIRM, 2024).
The first author was supported by the National Research Foundation of Korea Grant funded by the Korean Government (RS-2022-NR072395). 
The second author is partially supported by JSPS KAKENHI Grant Number JP23KJ1938 and JP23K12971.

\bibliographystyle{alpha}
\bibliography{biblio}

\end{document}

%% file: header_newcommand.tex
\newcommand{\HH}{\mathbb{H}}

\newcommand{\NN}{\mathbb{N}}

\newcommand{\RR}{\mathbb{R}}

\newcommand{\ZZ}{\mathbb{Z}}

\newcommand{\cD}{\mathcal{D}}

\newcommand{\cL}{\mathcal{L}}

\newcommand{\cN}{\mathcal{N}}
\newcommand{\cO}{\mathcal{O}}

\newcommand{\cU}{\mathcal{U}}

\newcommand{\fD}{\mathsf{D}}

\newcommand{\fM}{\mathsf{M}}
\newcommand{\fN}{\mathsf{N}}

\newcommand{\RP}{\operatorname{\RR P}}

\newcommand{\Diff}{\operatorname{Diff}}

\newcommand{\supp}{\operatorname{supp}}

\newcommand{\tDiff}{\widetilde{\Diff}}
\newcommand{\trot}{\widetilde{rot}}
\newcommand{\tgamma}{\widetilde{\gamma}}

\newcommand{\id}{\operatorname{Id}}

\newcommand{\closure}[1]{\overline{#1}}
\newcommand{\interior}{\operatorname{Int}}

\newcommand{\im}{\operatorname{Im}}

\newcommand{\near}{\operatorname{near}}

\newcommand{\eu}{\mathrm{eu}}
\newcommand{\dv}{\mathrm{div}}

\newcommand{\G}{G}
\newcommand{\Gn}{G_{\partial}}
\newcommand{\Gr}{G_{\mathrm{rel}}}

\newcommand{\Flux}{\operatorname{Flux}}
\newcommand{\Cal}{\operatorname{Cal}}
\newcommand{\Fluxn}{\Flux_{\partial}}
\newcommand{\Fluxnl}[1][\lambda]{\Flux_{\partial,#1}}